\documentclass [12pt]{amsart}
\usepackage{geometry}
\usepackage{graphicx}
\usepackage{amsthm}
\usepackage{amssymb}

\theoremstyle{plain}
\newtheorem{thm}{Theorem}[section]
\newtheorem{cor}[thm]{Corollary}
\newtheorem{lem}[thm]{Lemma}
\newtheorem{prop}[thm]{Proposition}
\newtheorem{defn}[thm]{Definition}
\newtheorem{exa}[thm]{Example}
\newtheorem{rem}[thm]{Remark}

\begin{document}

\title [Graded $r$-Submodules]{Graded $r$-Submodules }

\author[T. ALraqad]{Tariq Alraqad}

\address
{Tariq Alraqad, Department of Mathematics, University of Ha'il, Saudi Arabia.}
\email{t.alraqad@uoh.edu.sa}

\author[H. Saber]{Hicham Saber}

\address
{Hicham Saber, Department of Mathematics, University of Ha'il, Saudi Arabia.}
\email{hicham.saber7@gmail.com}

\author[R. Abu-Dawwas]{Rashid Abu-Dawwas}

\address
{Rashid Abu-Dawwas, Department of Mathematics, Yarmouk University, Jordan.}
\email{rrashid@yu.edu.jo}


\date{}

\begin{abstract}

Let $G$ be a group with identity $e$ and $R$ a commutative $G$-graded ring with a nonzero unity $1$. In this article, we introduce the concepts of graded $r$-submodules and graded special $r$-submodules, which are generalizations for the notion of graded r-ideals. For a nonzero $G$-graded $R$-module $M$, a  proper graded $R$-submodule $K$ of $M$ is said to be graded $r$-submodule (resp., graded special $r$-submodule) if whenever $a\in h(R)$ and $x\in h(M)$ such that $ax\in K$ with $Ann_{M}(a)=\{0\}$ (resp., $Ann_{R}(x)=\{0\}$), then $x\in K$ (resp., $a\in (K:_{R}M)$). We study various properties of graded $r$-submodules and graded special $r$-submodules, and we give several illustration examples of these two new classes of graded modules.
\end{abstract}

 \maketitle

 \section{Introduction}

 The concept of graded prime ideals plays a key role in the theory of commutative graded rings, and it has been widely studied. For example, see (\cite{Dawwas 2}, \cite{Dawwas 1}, \cite{Dawwas Bataineh2}, \cite{Zoubi Dawwas}). Graded prime ideals have been firstly introduced and studied by M. Refai, M. Hailat and S. Obiedat in \cite{Refai Hailat Obiedat}. A proper graded ideal $P$ of a graded ring $R$ is said to graded prime if whenever $x, y\in h(R)$ such that $xy\in P$, then either $x\in P$ or $y\in P$. Graded $r$-ideals have been introduced and studied in \cite{Dawwas Bataineh}. A proper graded ideal $P$ of a graded ring $R$ is said to be graded $r$-ideal if whenever $x, y\in h(R)$ such that $xy\in P$ and $Ann(x)=\{0\}$, then $y\in P$. In this article, we give two different generalizations of the concept of graded $r$-ideals to graded submodules by graded $r$-submodules and graded special $r$-submodules.

Graded prime submodules have been introduced and studied by S. E. Atani in \cite{Atani}. A proper graded $R$-submodule $N$ of $M$ is said to be graded prime if whenever $a\in h(R)$ and $x\in h(M)$ such that $ax\in N$, then either $x\in N$ or $a\in (N:_{R}M)$. In this case $(N:_{R}M)$ is a graded prime ideal of $R$. The concept of graded prime submodules has been widely studied in \cite{Abu-Dawwas4} and \cite{Abu-Dawwas3}. In \cite{Atani}, Atani defined graded pure submodules as a proper graded $R$-submodule $K$ of $M$ satisfies $aM\bigcap K=aK$ for all $a\in h(R)$. Our article is organized as follows:

In Section Two, we recall some standard results from the theory of graded rings and graded modules that will be used in the sequel.

In Section Three, we introduce and study the concept of graded $r$-submodules. We define a proper graded $R$-submodule $K$ of a graded $R$-module $M$ a graded $r$-submodule if whenever $a\in h(R)$ and $x\in h(M)$ such that $ax\in K$ with $Ann_{M}(a)=\{0\}$, then $x\in K$. Several results have been introduced; for example, we prove that graded $r$-submodules and graded prime submodules are totally different (Example \ref{Example 2.4.1} and Example \ref{Example 2.4.2}). On the other hand, we prove that if $K$ is a graded prime submodule of $M$, then $K$ is a graded $r$-submodule of $M$ if and only if $h(R)\cap(K:_{R}M)\subseteq Z(M)$ (Theorem \ref{Proposition 2.2}). We prove that if $K$ is a graded $r$-submodule of $M$, then $(K:_{R}M)$ is not graded $r$-ideal of $R$ in general (Example \ref{Example 2.6}). Also, We show that every graded pure submodule is a graded $r$-submodule (Theorem \ref{Proposition 2.4.2}), and we prove that the converse is not true in general (Example \ref{Example 2.7}). We prove that every proper graded $R$-submodule of $M$ is a graded $r$-submodule of $M$ if and only if $aK=K$ for every graded $R$-submodule $K$ of $M$ and for all $a\in h(R)-Z(M)$ (Theorem \ref{Theorem 2.5}). Also, we prove that if $K$ and $N$ are graded $r$-submodules of $M$ and $P$ is a graded ideal of $R$ such that $P\bigcap(h(R)-Z(M))\neq\emptyset$ and $PK=PN$, then $K=N$ (Theorem \ref{Theorem 2.2}).

In Section Four, we introduce the concept of graded special $r$-submodules which is another generalization of the concept of graded $r$-ideals. We define a proper graded $R$-submodule $K$ of a graded $R$-module $M$ a graded special $r$-submodule if whenever $a\in h(R)$ and $x\in h(M)$ such that $ax\in K$ with $Ann_{R}(x)=\{0\}$, then $a\in (K:_{R}M)$. Several results have been introduced; for example, we prove that the concepts of graded $r$-submodules and graded special $r$-submodules are different (Example \ref{Example 3.11} and Example \ref{Example 3.11.1}). We prove that if $K$ is a graded special $r$-submodule of $M$, then $h(M)\cap K\subseteq T(M)$ (Theorem \ref{Lemma 3.3}), and we prove that the converse is not true in general (Example \ref{Example 3.12}). We show that the concepts of graded prime $R$-submodules and graded special $r$-submodules are different (Example \ref{Example 3.13} and Example \ref{Example 3.13.1}). On the other hand, we prove that if $K$ is a graded prime $R$-submodule of $M$, then $K$ is a graded special $r$-submodule of $M$ if and only if $h(M)\cap K\subseteq T(M)$ (Theorem \ref{Proposition 3.10}), and we show that if $K$ is a graded maximal special $r$-submodule of a graded $R$-module $M$, then $K$ is a graded prime $R$-submodule of $M$ (Theorem \ref{Proposition 3.14}). Theorem \ref{Proposition 3.12} gives a nice characterization for graded special $r$-submodules. Also, we prove that if $K$ is a proper graded $R$-submodule of a graded $R$-module $M$, then $K$ is a graded special $r$-submodule of $M$ if and only if whenever $N$ is a graded $R$-submodule of $M$ such that $N\bigcap(h(M)-T(M))\neq\emptyset$ and $I$ is a graded ideal of $R$ such that $IN\subseteq K$, then $I\subseteq(K:_{R}M)$ (Theorem \ref{Theorem 3.10}). We prove that if $\{0\}$ is the only graded special $r$-submodule of $M$ and $M$ is homogeneous faithful, then $M$ is homogeneous torsion free (Theorem \ref{Theorem 3.13 (2)}). Also, we show that every proper graded $R$-submodule of $M$ is graded special $r$-submodule if and only if $h(M)\subseteq T(M)$ or $Rx=M$ for all $x\in h(M)-T(M)$ (Theorem \ref{Theorem 3.14}).


\section{Preliminaries}

 Throughout this article, $R$ is assumed to be a commutative ring with a nonzero unity $1$.
 Let $G$ be a group with identity $e$. A ring $R$ is said to be $G$-graded ring if there
 exist additive subgroups $R_{g}$ of $R$ such that $R=\displaystyle\bigoplus_{g\in G}R_{g}$ and
 $R_{g}R_{h}\subseteq R_{gh}$ for all $g,h\in G$. The elements of $R_{g}$ are called homogeneous of
 degree $g$ and $R_{e}$ (the identity component of $R$) is a subring of $R$ with $1\in R_{e}$.
 For $x\in R$, $x$ can be written uniquely as $\displaystyle\sum_{g\in G}x_{g}$ where $x_{g}$ is
 the component of $x$ in $R_{g}$. Also, we write $h(R)=\displaystyle\bigcup_{g\in G}R_{g}$ and $supp(R,G)=\left\{g\in G:R_{g}\neq0\right\}$.

 Let $R$ be a $G$-graded ring and $I$ be an ideal of $R$. Then $I$ is called $G$-graded ideal if $I=\displaystyle\bigoplus_{g\in G}\left(I\bigcap R_{g}\right)$, i.e., if $x\in I$
 and $x=\displaystyle\sum_{g\in G}x_{g}$, then $x_{g}\in I$ for all $g\in G$. An ideal of a graded ring need not be graded; see the following example.

 \begin{exa} \label{1}Consider $R=\textbf{Z}[i]$ and $G=\textbf{Z}_{2}$. Then $R$ is $G$-graded by $R_{0}=\textbf{Z}$ and $R_{1}=i\textbf{Z}$. Now, $P=\langle1+i\rangle$ is an ideal of $R$ with $1+i\in P$. If $P$ is a graded ideal, then $1\in P$, so $1=a(1+i)$ for some $a\in R$, i.e., $1=(x+iy)(1+i)$ for some $x,y\in \textbf{Z}$. Thus $1=x-y$ and $0=x+y$, i.e., $2x=1$ and hence $x=\frac{1}{2}$ a contradiction. So, $P$ is not a graded ideal of $R$.\end{exa}

 Let $M$ be a nonzero left $R$ - module. Then $M$ is a $G$-graded $R$-module if there exist additive subgroups $M_{g}$ of $M$ indexed by the elements $g\in G$ such that
$M=\displaystyle\bigoplus_{g\in G}M_{g}$ and $R_{g}M_{h}\subseteq M_{gh}$ for all $g,h\in G$. The elements of $M_{g}$ are called homogeneous of degree $g$. If $x\in M$, then $x$ can be written uniquely as $\displaystyle\sum_{g\in G}x_{g}$, where $x_{g}$ is the component of $x$ in $M_{g}$. Clearly, $M_{g}$ is $R_{e}$-submodule of $M$ for all $g\in G$. Also, we write
$h(M)=\displaystyle\bigcup_{g\in G}M_{g}$ and $supp(M,G)=\left\{g\in G:M_{g}\neq0\right\}$.

Let $M$ be a $G$-graded $R$-module and $N$ be an $R$-submodule of $M$. Then $N$ is called $G$-graded $R$-submodule if $N=\displaystyle\bigoplus_{g\in
G}\left(N\bigcap M_{g}\right)$, i.e., if $x\in N$ and $x=\displaystyle\sum_{g\in G}x_{g}$, then $x_{g}\in N$ for all $g\in G$. Not all $R$-submodules of a $G$-graded $R$-module are $G$-graded (Example \ref{1} will be helpful). For more details in this terminology, see \cite{Nastasescue}.

\begin{lem}\label{B}(\cite{Farzalipour}, Lemma 2.1) Let $R$ be a $G$-graded ring and $M$ be a $G$-graded $R$-module.

\begin{enumerate}

\item If $I$ and $J$ are graded ideals of $R$, then $I+J$ and $I\bigcap J$ are graded ideals of $R$.

\item If $N$ and $K$ are graded $R$-submodules of $M$, then $N+K$ and $N\bigcap K$ are graded $R$-submodules of $M$.

\item If $N$ is a graded $R$-submodule of $M$, $r\in h(R)$, $x\in h(M)$ and $I$ is a graded ideal of $R$, then $Rx$, $IN$ and $rN$ are graded $R$-submodules of $M$. Moreover, $(N:_{R}M)=\left\{r\in R:rM\subseteq N\right\}$ is a graded ideal of $R$.
\end{enumerate}
\end{lem}

Moreover, it has been proved in \cite{Farzalipour2} that if $N$ is a graded $R$-submodule of $M$, then $Ann_{R}(N)=\left\{r\in R:rN=\{0\}\right\}$ is a graded ideal of $R$. Similarly, if $I$ is a graded ideal of $R$, then $Ann_{M}(I)=\left\{m\in M:Im=\{0\}\right\}$ is a graded $R$-submodule of $M$. The set $\left\{a\in h(R):Ann_{M}(a)\neq\{0\}\right\}$ will be denoted by $Z(M)$ and the set $\left\{x\in M:Ann_{R}(x)\neq\{0\}\right\}$ will be denoted by $T(M)$.

A graded $R$-module $M$ is said to be homogeneous torsion free if whenever $r\in h(R)$ and $x\in h(M)$ such that $rx=0$, then either $r=0$ or $x=0$. Also, a graded $R$-module $M$ is said to be homogeneous faithful if $rM\neq\{0\}$ for all $0\neq r\in h(R)$.

Let $M$ be a $G$-graded $R$-module and $N$ be an $G$-graded $R$-submodule of $M$. Then $M/N$ may be made into a graded module by putting $(M/N)_{g}=(M_{g}+N)/N$ for all $g\in G$ (see \cite{Nastasescue}).

The following basic facts will be needed, and it is nice to go through proofs to get wet before diving.

\begin{thm} Let $M$ be a graded $R$-module, $N$ be an $R$-submodules of $M$, and $K$ be a graded $R$-submodule of $M$ such that $K\subseteq N$. Then $N$ is a graded $R$-submodule of $M$ if and only if $N/K$ is a graded $R$-submodule of $M/K$.
 \end{thm}

 \begin{proof} Suppose that $N$ is a graded $R$-submodule of $M$. Clearly, $N/K$ is an $R$-submodule of $M/K$. Let $x+K\in N/K$. Then $x\in N$ and since $N$ is graded, $x=\displaystyle\sum_{g\in G}x_{g}$ where $x_{g}\in N$ for all $g\in G$ and then $(x+K)_{g}=x_{g}+K\in N/K$ for all $g\in G$. Hence, $N/K$ is a graded $R$-submodule of $M/K$. Conversely, let $x\in N$. Then $x=\displaystyle\sum_{g\in G}x_{g}$ where $x_{g}\in M_{g}$ for all $g\in G$ and then $(x_{g}+K)\in (M_{g}+K)/K=\left(M/K\right)_{g}$ for all $g\in G$ such that \begin{center}$\displaystyle\sum_{g\in G}(x+K)_{g}=\displaystyle\sum_{g\in G}(x_{g}+K)=\left(\displaystyle\sum_{g\in G}x_{g}\right)+K=x+K\in N/K$.\end{center} Since $N/K$ is graded, $x_{g}+K\in N/K$ for all $g\in G$ which implies that $x_{g}\in N$ for all $g\in G$. Hence, $N$ is a graded $R$-submodule of $M$.
 \end{proof}

Let $M$ and $M^{\prime}$ be two $G$-graded $R$-modules. An $R$-homomorphism $f:M\rightarrow M^{\prime}$ is said to be graded $R$-homomorphism if $f(M_{g})\subseteq M^{\prime}_{g}$ for all $g\in G$ (see \cite{Nastasescue}).

\begin{thm} Suppose that $f:M\rightarrow M^{\prime}$ is a graded $R$-homomorphism.
\begin{enumerate}
\item If $f$ is a graded $R$-monomorphism and $K$ is a graded $R$-submodule of $M^{\prime}$, then $f^{-1}(K)$ is a graded $R$-submodule of $M$.
\item If $f$ is a graded $R$-epimorphism and $L$ is a graded $R$-submodule of $M$ with $Ker(f)\subseteq L$, then $f(L)$ is a graded $R$-submodule of $M^{\prime}$.
\end{enumerate}
\end{thm}

\begin{proof}
\begin{enumerate}
\item Clearly, $f^{-1}(K)$ is an $R$-submodule of $M$. Let $x\in f^{-1}(K)$. Then $x\in M$ with $f(x)\in K$, and then $x=\displaystyle\sum_{g\in G}x_{g}$ where $x_{g}\in M_{g}$ for all $g\in G$. So, for every $g\in G$, $f(x_{g})\in f(M_{g})\subseteq M^{\prime}_{g}$ such that $\displaystyle\sum_{g\in G}f(x_{g})=f\left(\displaystyle\sum_{g\in G}x_{g}\right)=f(x)\in K$. Since $K$ is graded, $f(x_{g})\in K$ for all $g\in G$, i.e., $x_{g}\in f^{-1}(K)$ for all $g\in G$. Hence, $f^{-1}(K)$ is a graded $R$-submodule of $M$.
\item Clearly, $f(L)$ is an $R$-submodule of $M^{\prime}$. Let $y\in f(L)$. Then $y\in M^{\prime}$, and so there exists $x\in M$ such that $y=f(x)$, so $f(x)\in L$, which implies that $x\in L$ since $Ker(f)\subseteq L$, and hence $x_{g}\in L$ for all $g\in G$ since $L$ is graded. Thus, $y_{g}=(f(x))_{g}=f(x_{g})\in f(L)$ for all $g\in G$. Therefore, $f(L)$ is a graded $R$-submodule of $M^{\prime}$.
\end{enumerate}
\end{proof}


\section{Graded $r$-Submodules}

In this section, we introduce and study the concept of graded $r$-submodules.

\begin{defn} Let $M$ be a graded $R$-module. A proper graded $R$-submodule $K$ of $M$ is said to be graded $r$-submodule if whenever $a\in h(R)$ and $x\in h(M)$ such that $ax\in K$ with $Ann_{M}(a)=\{0\}$, then $x\in K$.
\end{defn}

\begin{rem} One can clearly see that proper graded $R$-submodule $K$ of $M$ is graded $r$-submodule is equivalent to the facto that $Z(M/K)\subseteq Z(M)$. Also, the graded $r$-submodules of the graded $R$-module $R$ are exactly the graded $r$-ideals of $R$. Note that $Z(M/\{0\})=Z(M)$, and so $\{0\}$ is a graded $r$-submodule of $M$.
\end{rem}

\begin{exa}\label{example 2.1} Consider $R=\mathbb{Z}$, $G=\mathbb{Z}_{2}$ and $M=\mathbb{Z}_{n}[i]=\left\{a+ib:a, b\in \mathbb{Z}_{n}\right\}$. Then $R$ is $G$-graded by $R_{0}=\mathbb{Z}$ and $R_{1}=\{0\}$. Also, $M$ is $G$-graded by $M_{0}=\mathbb{Z}_{n}$ and $M_{1}=i\mathbb{Z}_{n}$. Consider the graded $R$-submodule $K=\langle m\rangle$ of $M$, where $m\in \mathbb{Z}_{n}$ such that $gcd(m, n)=s>1$. Then $K=\langle s\rangle$. Also, $\mathbb{Z}_{n}[i]/K\cong\mathbb{Z}_{s}[i]$. Since $Z(\mathbb{Z}_{s}[i])\subseteq Z(\mathbb{Z}_{n}[i])$, $K$ is a graded $r$-submodule of $M$.
\end{exa}

The following example  shows that a graded $r$-submodule need not be a graded prime submodule.

\begin{exa}\label{Example 2.4.1} Consider $R=\mathbb{Z}$, $G=\mathbb{Z}_{2}$ and $M=\mathbb{Z}_{18}[i]=\left\{a+ib:a, b\in \mathbb{Z}_{18}\right\}$. Then $R$ is $G$-graded by $R_{0}=\mathbb{Z}$ and $R_{1}=\{0\}$. Also, $M$ is $G$-graded by $M_{0}=\mathbb{Z}_{18}$ and $M_{1}=i\mathbb{Z}_{18}$. Consider the graded $R$-submodule $K=\langle9\rangle$ of $M$. By Example \ref{example 2.1}, $K$ is a graded $r$-submodule of $M$, but $K$ is not a graded prime $R$-submodule of $M$ since $3\in h(R)$ and $3\in h(M)$ such that $3.3\in K$ but $3\notin K$ and $3\notin (K:_{R}M)=9\mathbb{Z}$.
\end{exa}

Also, the next example shows that a graded prime submodule need not be a graded $r$-submodule. So, the concepts of graded $r$-submodules and graded prime submodules are  different.

\begin{exa}\label{Example 2.4.2} Consider $R=\mathbb{Z}$, $G=\mathbb{Z}_{2}$ and $M=\mathbb{Z}[i]=\left\{a+ib:a, b\in \mathbb{Z}\right\}$. Then $R$ is $G$-graded by $R_{0}=\mathbb{Z}$ and $R_{1}=\{0\}$. Also, $M$ is $G$-graded by $M_{0}=\mathbb{Z}$ and $M_{1}=i\mathbb{Z}$. Consider the graded $R$-submodule $K=\langle3\rangle$ of $M$. Clearly, $K$ is a graded prime $R$-submodule of $M$, but $K$ is not a graded $r$-submodule of $M$ since $(K:_{R}M)=3\mathbb{Z}\nsubseteq Z(M)$.
\end{exa}

\begin{lem}\label{Lemma 2.1} Let $M$ be a graded $R$-module and $K$ a graded $r$-submodule of $M$. Then $h(R)\cap (K:_{R}M)\subseteq Z(M)$.
\end{lem}

\begin{proof} Since $K$ is a graded $r$-submodule of $M$, we have  $Z(M/K)\subseteq Z(M)$. Therefore $(K:_{R}M)=Ann_{R}(M/K)\subseteq Z(M/K)\subseteq Z(M)$.
\end{proof}

 The converse of Lemma \ref{Lemma 2.1} need  not to be  true in general. Indeed, we have :

\begin{exa}\label{Example 2.3} Consider $R=\mathbb{Z}$, $G=\mathbb{Z}_{2}$ and $M=\mathbb{Z}\times\mathbb{Z}$. Then $R$ is $G$-graded by $R_{0}=\mathbb{Z}$ and $R_{1}=\{0\}$. Also, $M$ is $G$-graded by $M_{0}=\mathbb{Z}\times\{0\}$ and $M_{1}=\{0\}\times\mathbb{Z}$. Consider the graded $R$-submodule $K=2\mathbb{Z}\times\{0\}$ of $M$. Clearly, $(K:_{R}M)=\{0\}\subseteq Z(M)$. Also, $M/K\cong\mathbb{Z}_{2}\times\mathbb{Z}$. Since $2\in Z(\mathbb{Z}_{2}\times\mathbb{Z})$ and  $2\notin Z(M)$, $Z(\mathbb{Z}_{2}\times\mathbb{Z})\nsubseteq Z(M)$. Thus $K$ is not a graded $r$-submodule of $M$.
\end{exa}

\begin{thm}\label{Proposition 2.2} Let $M$ be a graded $R$-module and $K$ a graded prime $R$-submodule of $M$. Then $K$ is a graded $r$-submodule of $M$ if and only if $h(R)\cap(K:_{R}M)\subseteq Z(M)$.
\end{thm}

\begin{proof} Suppose that $h(R)\cap(K:_{R}M)\subseteq Z(M)$. Since $K$ is graded prime, $Z(M/K)=(K:_{R}M)\subseteq Z(M)$ and then $K$ is a graded $r$-submodule of $M$. The converse holds by Lemma \ref{Lemma 2.1}.
\end{proof}

\begin{prop}\label{Proposition 2.3} Let $M$ be a graded $R$-module, $K_{1}$ and $K_{2}$ a graded $r$-submodules of $M$. Then $K_{1}\bigcap K_{2}$ is a graded $r$-submodule of $M$.
\end{prop}

\begin{proof} Let $a\in h(R)$ and $x\in h(M)$ such that $ax\in K_{1}\bigcap K_{2}$ with $Ann_{M}(a)=\{0\}$.  Since $K_{1}$ and $K_{2}$ are graded $r$-submodules of $M$, and  $ax\in K_{1}$ and $ax\in K_{2}$, we get  $x\in K_{1}\bigcap K_{2}$. Thus, $K_{1}\bigcap K_{2}$ is a graded $r$-submodule of $M$.
\end{proof}

It has been proved in (\cite{Atani}, Proposition 2.7) that if $K$ is a graded prime $R$-submodule of $M$, then $(K:_{R}M)$ is a graded prime ideal of $R$. The next example shows this is not true for graded $r$-submodules. Let us first recall the proof of the following theorem \cite{Dawwas Bataineh}.

\begin{thm}\label{A}(\cite{Dawwas Bataineh}) If $R$ is a $G$-graded domain, then $\{0\}$ is the unique graded $r$-ideal of $R$.
\end{thm}

\begin{proof} Let $P$ be a nonzero proper graded ideal of $R$. Then there exists $0\neq a=\displaystyle\sum_{g\in G}a_{g}\in P$ with $a_{g}\in P$ for all $g\in G$. Since $1.a_{g}\in P$  and $Ann(a_{g})=\{0\}$, $R$ is a domain, then from the fact that $P$ is a graded $r$-ideal, we get $1\in P$, a contradiction. Hence, $\{0\}$ is the only graded $r$-ideal of $R$.
\end{proof}

\begin{exa}\label{Example 2.6} Consider $R=\mathbb{Z}$, $G=\mathbb{Z}_{2}$ and $M=\mathbb{Z}_{4}[i]=\left\{a+ib:a, b\in \mathbb{Z}_{4}\right\}$. Then $R$ is $G$-graded by $R_{0}=\mathbb{Z}$ and $R_{1}=\{0\}$. Also, $M$ is $G$-graded by $M_{0}=\mathbb{Z}_{4}$ and $M_{1}=i\mathbb{Z}_{4}$. By Example \ref{example 2.1}, $K=\langle2\rangle$ is a graded $r$-submodule of $M$, but $(K:_{R}M)=2\mathbb{Z}$ is not a graded $r$-ideal of $R$, since a domain has no nonzero graded $r$-ideals by Theorem \ref{A}.
\end{exa}

\begin{thm}\label{Proposition 2.4.1} Let $M$ be a graded $R$-module and $K$ a proper graded $R$-submodule of $M$. Then $K$ is a graded $r$-submodule of $M$ if and only if $aM\bigcap K=aK$ for all $a\in h(R)-Z(M)$.
\end{thm}

\begin{proof} Suppose that $K$ is a graded $r$-submodule of $M$. Let $a\in h(R)-Z(M)$. Clearly, $aK\subseteq aM\bigcap K$. Let $x\in aM\bigcap K$. Since $x\in K$ and $K$ is a graded $R$-submodule of $M$, $x_{g}\in K$ for all $g\in G$. Also, since $x\in aM$ and $aM$ is a graded $R$-submodule of $M$ (by Lemma \ref{B}), $x_{g}\in aM$ for all $g\in G$. So, for every $g\in G$, $x_{g}\in aM\bigcap K$ which implies that $x_{g}=am_{g}\in K$ for some $m_{g}\in M$. In fact, $m_{g}\in h(M)$ since $x_{g}\in h(M)$ and $a\in h(R)$. Now, $K$ is a graded $r$-submodule of $M$ implies that $m_{g}\in K$, then $x_{g}=am_{g}\in aK$ for all $g\in G$ and hence  $x=\displaystyle\sum_{g\in G}x_{g}\in aK$. Thus $aM\bigcap K=aK$. Conversely, let $a\in h(R)$ and $x\in h(M)$ such that $ax\in K$ with $Ann_{M}(a)=\{0\}$. Then $ax\in aM\bigcap K=aK$ which implies that $ax=ak$ for some $k\in K$. Since $Ann_{M}(a)=\{0\}$, $x=k\in K$. Hence, $K$ is a graded $r$-submodule of $M$.
\end{proof}

\begin{cor}\label{Proposition 2.4.2} Let $M$ be a graded $R$-module. Then every graded pure $R$-submodule of $M$ is a graded $r$-submodule of $M$.
\end{cor}

\begin{proof} Apply Theorem \ref{Proposition 2.4.1}.
\end{proof}

The next example shows that the converse of Corollary \ref{Proposition 2.4.2} is not true in general.

\begin{exa}\label{Example 2.7} Consider $R=\mathbb{Z}$, $G=\mathbb{Z}_{2}$ and $M=\mathbb{Z}_{16}[i]=\left\{a+ib:a, b\in \mathbb{Z}_{16}\right\}$. Then $R$ is $G$-graded by $R_{0}=\mathbb{Z}$ and $R_{1}=\{0\}$. Also, $M$ is $G$-graded by $M_{0}=\mathbb{Z}_{16}$ and $M_{1}=i\mathbb{Z}_{16}$. By Example \ref{example 2.1}, $K=\langle2\rangle$ is a graded $r$-submodule of $M$. On the other hand, $2\in h(R)$ such that $2M\bigcap K=K\neq 2K$, which means that $K$ is not a graded pure $R$-submodule of $M$.
\end{exa}

\begin{thm}\label{Theorem 2.5} Let $M$ be a graded $R$-module. Then every proper graded $R$-submodule of $M$ is a graded $r$-submodule of $M$ if and only if $aK=K$ for every graded $R$-submodule $K$ of $M$ and for all $a\in h(R)-Z(M)$.
\end{thm}

\begin{proof} Suppose that every proper graded $R$-submodule of $M$ is a graded $r$-submodule of $M$. Let $K$ be a graded $R$-submodule of $M$ and $a\in h(R)-Z(M)$. Assume that $K=M$. If $aM\neq M$, then $aM$ is a graded $r$-submodule of $M$ by assumption. Let $x\in M$. Then $x=\displaystyle\sum_{g\in G}x_{g}$ where $x_{g}\in M_{g}$ for all $g\in G$. Now, for every $g\in G$, $ax_{g}\in aM$, and then $x_{g}\in aM$ for all $g\in G$ since $aM$ is a graded $r$-submodule of $M$, and hence $x\in aM$. So, $aM=M$ which is a contradiction. Therefore, $aM=M$. Let $K$ be a proper graded $R$-submodule of $M$. Then $aK\subseteq K\neq M$, and so $aK$ is a graded $r$-submodule of $M$ by assumption. Let $k\in K$. We have $k_{g}\in K$ for all $g\in G$ since $K$ is graded, and then for every $g\in G$, $ak_{g}\in aK$, which implies that $k_{g}\in aK$ for all $g\in G$ since $aK$ is a graded $r$-submodule of $M$. Thus, $k=\displaystyle\sum_{g\in G}k_{g}\in aK$, and so $aK=K$. Conversely, let $K$ be a proper graded $R$-submodule of $M$ and $a\in h(R)-Z(M)$. Then $aM\bigcap K=aM\bigcap aK=aK$, and then $K$ is a graded $r$-submodule of $M$ by Theorem \ref{Proposition 2.4.1}.
\end{proof}

\begin{lem}\label{Proposition 2.4.3} Let $M$ be a graded $R$-module, $K$ a graded $R$-submodule of $M$ and $a\in h(R)$. Then $(K:_{M}a)=\left\{m\in M:am\in K\right\}$ is a graded $R$-submodule of $M$.
\end{lem}

\begin{proof} Clearly, $(K:_{M}a)$ is an $R$-submodule of $M$. Let $m\in (K:_{M}a)$. Then $m\in M$ such that $am\in K$.  $M$ is graded implies that $m=\displaystyle\sum_{g\in G}m_{g}$ where $m_{g}\in M_{g}$ for all $g\in G$. Since $a\in h(R)$ and $m_{g}\in h(M)$ for all $g\in G$, $am_{g}\in h(M)$ for all $g\in G$ with $\displaystyle\sum_{g\in G}am_{g}=a\left(\displaystyle\sum_{g\in G}m_{g}\right)=am\in K$, and from the fact that $K$ is a graded $R$-submodule of $M$, we get  $am_{g}\in K$ for all $g\in G$, which means that $m_{g}\in (K:_{M}a)$ for all $g\in G$. Hence, $(K:_{M}a)$ is a graded $R$-submodule of $M$.
\end{proof}

\begin{thm}\label{Proposition 2.4.4} Let $M$ be a graded $R$-module and $K$ a proper graded $R$-submodule of $M$. Then $K$ is a graded $r$-submodule of $M$ if and only if $(K:_{M}a)=K$ for all $a\in h(R)-Z(M)$.
\end{thm}

\begin{proof} Suppose that $K$ is a graded $r$-submodule of $M$. Let $a\in h(R)-Z(M)$. Clearly, $K\subseteq(K:_{M}a)$. Assume that $x\in (K:_{M}a)$. Since $(K:_{M}a)$ is a graded $R$-submodule of $M$ (by Lemma \ref{Proposition 2.4.3}), $x_{g}\in (K:_{M}a)$ for all $g\in G$, and then $x_{g}\in h(M)$ such that $ax_{g}\in K$ for all $g\in G$. Therefore  $x_{g}\in K$ for all $g\in G$, $K$ is a graded $r$-submodule of $M$. and hence $x=\displaystyle\sum_{g\in G}x_{g}\in K$. Thus $(K:_{M}a)=K$ as desired. Conversely, let $a\in h(R)$ and $x\in h(M)$ such that $ax\in K$ with $Ann_{M}(a)=\{0\}$. Then $x\in (K:_{M}a)=K$, and hence $K$ is a graded $r$-submodule of $M$.
\end{proof}

\begin{lem}\label{Corollary 2.1} Let $M$ be a graded $R$-module and $a\in h(R)-Ann_{R}(M)$. Then $Ann_{M}(a)$ is a graded $r$-submodule of $M$.
\end{lem}

\begin{proof} Clearly, $Ann_{M}(a)$ is a proper graded $R$-submodule of $M$. Let $b\in h(R)$ and $x\in h(M)$ such that $bx\in Ann_{M}(a)$ with $Ann_{M}(b)=\{0\}$. Then $abx=0\in \{0\}$, and since $\{0\}$ is a graded $r$-submodule of $M$, $ax=0$, and then $x\in Ann_{M}(a)$. Hence, $Ann_{M}(a)$ is a graded $r$-submodule of $M$.
\end{proof}

\begin{thm}\label{Proposition 2.6} Let $M$ be a graded $R$-module such that $\{0\}$ is the only graded $r$-submodule of $M$. Then $\{0\}$ is a graded prime $R$-submodule of $M$.
\end{thm}

\begin{proof} Let $a\in h(R)$ and $x\in h(M)$ such that $ax=0$. Suppose that $a\notin Ann_{R}(M)$. Then by Lemma \ref{Corollary 2.1}, $Ann_{M}(a)$ is a graded $r$-submodule of $M$, and so $Ann_{M}(a)=\{0\}$ by assumption, hence $x=0$, Thus $\{0\}$ is a graded prime $R$-submodule of $M$.
\end{proof}

\begin{cor}\label{Proposition 2.6.1} Let $M$ be a graded $R$-module such that $\{0\}$ is the only graded $r$-submodule of $M$. Then $Ann_{R}(M)$ is a graded prime ideal of $R$.
\end{cor}

\begin{proof} It follows from Theorem \ref{Proposition 2.6}.
\end{proof}

\begin{lem}\label{Theorem 2.1} Let $M$ be a graded $R$-module and $K$ a proper graded $R$-submodule of $M$. Then $K$ is a graded $r$-submodule of $M$ if and only if whenever $P$ is a graded ideal of $R$ and $N$ a graded $R$-submodule of $M$ such that $PN\subseteq K$ with $P\bigcap(h(R)-Z(M))\neq\emptyset$, then $N\subseteq K$.
\end{lem}

\begin{proof} Suppose that $K$ is a graded $r$-submodule of $M$. Let $P$ be a graded ideal of $R$ and $N$ a graded $R$-submodule of $M$ such that $PN\subseteq K$ with $P\bigcap(h(R)-Z(M))\neq\emptyset$. Then there exists $a\in h(R)\bigcap P$ such that $Ann_{M}(a)=\{0\}$. Assume that $x\in N$. Since $N$ is graded, $x_{g}\in N$ for all $g\in G$. So, for every $g\in G$, $ax_{g}\in K$, and since $K$ is a graded $r$-submodule of $M$, $x_{g}\in K$ for all $g\in G$, and hence $x=\displaystyle\sum_{g\in G}x_{g}\in K$. Therefore, $N\subseteq K$. Conversely, let $a\in h(R)$ and $x\in h(M)$ such that $ax\in K$ with $Ann_{M}(a)=\{0\}$. Then $P=\langle a\rangle$ is a graded ideal of $R$, $N=\langle x\rangle$ is a graded $R$-submodule of $M$ such that $PN\subseteq K$ and $P\bigcap(h(R)-Z(M))\neq\emptyset$. Then $N\subseteq K$ by assumption, and so $x\in K$. Hence, $K$ is a graded $r$-submodule of $M$.
\end{proof}

\begin{thm}\label{Theorem 2.2} Let $M$ be a graded $R$-module, $K$ and $N$ a graded $r$-submodules of $M$ and $P$ a graded ideal of $R$ such that $P\bigcap(h(R)-Z(M))\neq\emptyset$. If $PK=PN$, then $K=N$.
\end{thm}

\begin{proof} Since $PK=PN\subseteq N$ and $N$ is a graded $r$-submodule of $M$, $K\subseteq N$ by Lemma \ref{Theorem 2.1}. Similarly, $N\subseteq K$, and then $K=N$.
\end{proof}

\begin{thm}\label{Theorem 2.2.1} Let $M$ be a graded $R$-module, $K$ a graded $R$-submodule of $M$ and $P$ a graded ideal of $R$ such that $P\bigcap(h(R)-Z(M))\neq\emptyset$. If $PK$ is a graded $r$-submodule of $M$, then $PK=K$ and hence $K$ is a graded $r$-submodule of $M$.
\end{thm}

\begin{proof} Since $PK\subseteq PK$ and $PK$ is a graded $r$-submodule of $M$, $K\subseteq PK\subseteq K$ by Lemma \ref{Theorem 2.1}. Hence, $PK=K$.
\end{proof}

\begin{thm}\label{Theorem 2.1.1} Let $M$ be a graded $R$-module and $K$ a proper graded $R$-submodule of $M$. If $(K:_{R}M)\subseteq Z(M)$ and $K$ is not a graded $r$-submodule of $M$, then there exist a graded ideal $P$ of $R$ and a graded $R$-submodule $N$ of $M$ such that $P\bigcap(h(R)-Z(M))\neq\emptyset$, $K\subsetneqq N$, $(K:_{R}M)\subsetneqq P$ and $PN\subseteq K$.
\end{thm}

\begin{proof} Since $K$ is not a graded $r$-submodule of $M$, there exist $a\in h(R)$ and $x\in h(M)$ such that $ax\in K$ with $Ann_{M}(a)=\{0\}$ and $x\notin K$. Then $P=(K:_{R}x)$ is a graded ideal of $R$ such that $a\in P$ and $a\notin (K:_{R}M)$ since $Ann_{M}(a)=\{0\}$, and hence $(K:_{R}M)\subsetneqq P$. Also, $N=(K:_{M}P)$ is a graded $R$-submodule of $M$ such that $x\notin K$ and $x\in N$, and hence $K\subsetneqq N$ and $PN=P(K:_{M}P)\subseteq K$.
\end{proof}


\section{Graded Special $r$-Submodules}

In this section, we introduce and study the concept of graded special $r$-submodules.

\begin{defn} Let $M$ be a graded $R$-module. A proper graded $R$-submodule $K$ of $M$ is said to be graded special $r$-submodule if whenever $a\in h(R)$ and $x\in h(M)$ such that $ax\in K$ with $Ann_{R}(x)=\{0\}$, then $a\in(K:_{R}M)$.
\end{defn}

\begin{exa}\label{Proposition 3.11} Let $M$ be a graded $R$-module. To prove that $\{0\}$ is a graded special $r$-submodule of $M$, assume that $a\in h(R)$ and $x\in h(M)$ such that $ax=0$ with $Ann_{R}(x)=\{0\}$. Then $a=0\in (\{0\}:_{R}M)$, and hence $\{0\}$ is a graded special $r$-submodule of $M$.
\end{exa}

The next two examples show that the concepts of graded $r$-submodules and graded special $r$-submodules are different.

\begin{exa}\label{Example 3.11} Consider $R=\mathbb{R}$ (field of real numbers), $G=\mathbb{Z}_{2}$ and $M=\mathbb{R}^{2}$. Then $R$ is $G$-graded by $R_{0}=\mathbb{R}$ and $R_{1}=\{0\}$. Also, $M$ is $G$-graded by $M_{0}=\mathbb{R}\times\{0\}$ and $M_{1}=\{0\}\times\mathbb{R}$. Consider the graded $R$-submodule $K=\left\{(m, 0):m\in \mathbb{R}\right\}$ of $M$. Since $Z(M/K)=\{0\}$, $K$ is a graded $r$-submodule of $M$. On the other hand, $2\in h(R)$ and $(1, 0)\in h(M)$ such that $2(1, 0)=(2, 0)\in K$ with $Ann_{R}((1, 0))=\{0\}$, but $2\notin (K:_{R}M)$, which means that $K$ is not a graded special $r$-submodule of $M$.
\end{exa}

\begin{exa}\label{Example 3.11.1} Consider $R=\mathbb{Z}\times\mathbb{Z}$, $G=\mathbb{Z}_{2}$ and $M=\mathbb{Z}\times\mathbb{Z}_{2}$. Then $R$ is $G$-graded by $R_{0}=\mathbb{Z}\times\mathbb{Z}$ and $R_{1}=\{0\}$. Also, $M$ is $G$-graded by $M_{0}=\mathbb{Z}\times\mathbb{Z}_{2}$ and $M_{1}=\{0\}$. Consider the graded $R$-submodule $K=2\mathbb{Z}\times\{0\}$ of $M$. Since $Ann_{R}(x)\neq\{0\}$ for all $x\in h(M)$, $K$ is a graded special $r$-submodule of $M$. On the other hand, $(2, 0)\in h(R)$ and $(1, 0)\in h(M)$ such that $(2, 0)(1, 0)=(2, 0)\in K$ with $Ann_{M}((2, 0))=\{0\}$, but $(1, 0)\notin K$, which means that $K$ is not a graded $r$-submodule of $M$.
\end{exa}

\begin{thm}\label{Lemma 3.3} Let $M$ be a graded $R$-module and $K$ a graded special $r$-submodule of $M$. Then $h(M)\cap K \subseteq T(M)$.
\end{thm}

\begin{proof} Suppose that $h(M)\cap K \nsubseteq T(M)$. Then there exists $k\in h(M)\cap K$ such that $Ann_{R}(k)=\{0\}$. Now, $1\in h(R)$ and
$k\in h(M)$ such that $1.k=k\in K$, and then $1\in (K:_{R}M)$ since $K$ is a graded special $r$-submodule of $M$, and hence $K=M$ which is a contradiction. Therefore, $(h(M)\cap K)\subseteq T(M)$.
\end{proof}

The following example shows that the converse of Theorem \ref{Lemma 3.3} is not true in general.

\begin{exa}\label{Example 3.12} Consider $R=\mathbb{R}\times\mathbb{Z}$, $G=\mathbb{Z}_{2}$ and $M=\mathbb{C}\times\mathbb{Z}$. Then $R$ is $G$-graded by $R_{0}=\mathbb{R}\times\mathbb{Z}$ and $R_{1}=\{0\}$. Also, $M$ is $G$-graded by $M_{0}=\mathbb{C}\times\mathbb{Z}$ and $M_{1}=\{0\}$. Consider the graded $R$-submodule $K=\mathbb{R}\times\{0\}$ of $M$. Clearly, $(K:_{R}M)=\{0\}$ and $T(M)=\left(\{0\}\times\mathbb{Z}\right)\bigcup\left(\mathbb{C}\times\{0\}\right)$. So, $h(M)\cap K\subseteq T(M)$. On the other hand, $(1, 0)\in h(R)$ and $(1, 1)\in h(M)$ such that $(1, 0)(1, 1)=(1, 0)\in K$ with $Ann_{R}((1, 1))=\{0\}$, but $(1, 0)\notin (K:_{R}M)$, which means that $K$ is not a graded special $r$-submodule of $M$.
\end{exa}

\begin{thm}\label{Theorem 3.13} Let $M$ be a graded $R$-module. If $T(M)=\{0\}$, then $\{0\}$ is the only graded special $r$-submodule of $M$.
\end{thm}

\begin{proof} Let $K$ be a graded special $r$-submodule of $M$. Then $h(M)\cap K\subseteq T(M)=\{0\}$ by Theorem \ref{Lemma 3.3}, which implies that $K=\{0\}$ which is a graded special $r$-submodule of $M$ by Example \ref{Proposition 3.11}.
\end{proof}

The next two examples show that the concepts of graded prime $R$-submodules and graded special $r$-submodules are different.

\begin{exa}\label{Example 3.13} Consider $R=\mathbb{Z}$, $G=\mathbb{Z}_{3}$ and $M=\mathbb{Z}$. Then $R$ is $G$-graded by $R_{0}=\mathbb{Z}$ and $R_{1}=R_{2}=\{0\}$. Also, $M$ is $G$-graded by $M_{0}=\mathbb{Z}$ and $M_{1}=M_{2}=\{0\}$. Consider the graded $R$-submodule $K=\langle2\rangle$ of $M$. Clearly, $K$ is a graded prime $R$-submodule of $M$ which is not a graded special $r$-submodule of $M$.
\end{exa}

\begin{exa}\label{Example 3.13.1} Consider $R=\mathbb{Z}$, $G=\mathbb{Z}_{3}$ and $M=\mathbb{Z}_{12}$. Then $R$ is $G$-graded by $R_{0}=\mathbb{Z}$ and $R_{1}=R_{2}=\{0\}$. Also, $M$ is $G$-graded by $M_{0}=\mathbb{Z}_{12}$ and $M_{1}=M_{2}=\{0\}$. Consider the graded $R$-submodule $K=\langle4\rangle$ of $M$. Clearly, $K$ is a graded special $r$-submodule of $M$ which is not a graded prime $R$-submodule of $M$.
\end{exa}

\begin{thm}\label{Proposition 3.10} Let $M$ be a graded $R$-module and $K$ a graded prime $R$-submodule of $M$. Then $K$ is a graded special $r$-submodule of $M$ if and only if $h(M)\cap K\subseteq T(M)$.
\end{thm}

\begin{proof} Suppose that $h(M)\cap K\subseteq T(M)$. Let $a\in h(R)$ and $x\in h(M)$ such that $ax\in K$ with $Ann_{R}(x)=\{0\}$. Then $x\notin T(M)$ which implies that $x\notin K$, and then $a\in (K:_{R}M)$ since $K$ is a graded prime $R$-submodule of $M$. Hence, $K$ is a graded special $r$-submodule of $M$. The converse holds from Theorem \ref{Lemma 3.3}.
\end{proof}

\begin{prop}\label{Proposition 3.11.1} Let $M$ be a graded $R$-module, $K_{1}$ and $K_{2}$ a graded special $r$-submodules of $M$. Then $K_{1}\bigcap K_{2}$ is a graded special $r$-submodule of $M$.
\end{prop}

\begin{proof} Let $a\in h(R)$ and $x\in h(M)$ such that $ax\in K_{1}\bigcap K_{2}$ with $Ann_{R}(x)=\{0\}$. Then $ax\in K_{1}$ and $ax\in K_{2}$, and so $a\in(K_{1}:_{R}M)\bigcap(K_{2}:_{R}M)=(K_{1}\bigcap K_{2}:_{R}M)$ since $K_{1}$ and $K_{2}$ are graded special $r$-submodules of $M$. Hence, $K_{1}\bigcap K_{2}$ is a graded special $r$-submodule of $M$.
\end{proof}

It has been proved in Example \ref{Example 2.6} that if $K$ is a graded $r$-submodule of $M$, then $(K:_{R}M)$ need not be a graded $r$-ideal of $R$. The next example shows that $(K:_{R}M)$ need not be a graded $r$-ideal of $R$ even if $K$ is a graded special $r$-submodule of $M$.

\begin{exa}\label{Example 3.14} Consider $R=\mathbb{Z}$, $G=\mathbb{Z}$ and $M=\mathbb{Z}_{6}[x]$. Then $R$ is $G$-graded by $R_{0}=\mathbb{Z}$ and $R_{j}=\{0\}$ otherwise. Also, $M$ is $G$-graded by $M_{j}=\mathbb{Z}_{6}x^{j}$ for $j\geq0$ and $M_{j}=\{0\}$ otherwise. Consider the graded $R$-submodule $K=\left\{p(x)\in M:p(0)\in \langle2\rangle\right\}$. Then $K$ is a graded special $r$-submodule of $M$, but $(K:_{R}M)=2\mathbb{Z}$ is not a graded $r$-ideal of $R$ by Theorem \ref{A}.
\end{exa}

\begin{thm}\label{Proposition 3.12} Let $K$ be a proper graded $R$-submodule of a graded $R$-module $M$. Then the following are equivalent.
\begin{enumerate}
\item $K$ is a graded special $r$-submodule of $M$.
\item $Rx\bigcap K=(K:_{R}M)x$ for all $x\in h(M)-T(M)$.
\item $(K:_{R}M)=(K:_{R}x)$ for all $x\in h(M)-T(M)$.
\end{enumerate}
\end{thm}

\begin{proof} $\underline{(1)\Rightarrow(2)}$: Let $x\in h(M)-T(M)$. Then clearly, $(K:_{R}M)x\subseteq Rx\bigcap K$. Assume that $y\in Rx\bigcap K$. Then $y=rx\in K$ fore some $r\in R$. Now, $r=\displaystyle\sum_{g\in G}r_{g}$ where $r_{g}\in R_{g}$ for all $g\in G$. Since $x\in h(M)$, $r_{g}x\in h(M)$ for all $g\in G$ such that $\displaystyle\sum_{g\in G}r_{g}x=\left(\displaystyle\sum_{g\in G}r_{g}\right)x=rx\in K$, and since $K$ is graded, $r_{g}x\in K$ for all $g\in G$.  $K$ is a graded special $r$-submodule of $M$ implies that $r_{g}\in (K:_{R}M)$ for all $g\in G$, and then $r=\displaystyle\sum_{g\in G}r_{g}\in (K:_{R}M)$. Thus $y=rx\in (K:_{R}M)x$, as desired.

$\underline{(2)\Rightarrow(3)}$: Let $x\in h(M)-T(M)$. Then clearly, $(K:_{R}M)\subseteq(K:_{R}x)$. Assume that $r\in (K:_{R}x)$. Then  $rx\in Rx\bigcap K=(K:_{R}M)x$ by assumption. So, $rx=sx$ for some $s\in (K:_{R}M)$, which implies that $(r-s)x=0$. Since $Ann_{R}(x)=\{0\}$, $r=s\in (K:_{R}M)$, as desired.

$\underline{(3)\Rightarrow(1)}$: Let $r\in h(R)$ and $x\in h(M)$ such that $rx\in K$ with $Ann_{R}(x)=\{0\}$. Then $x\in h(M)-T(M)$ with $r\in (K:_{R}x)=(K:_{R}M)$ by assumption. Hence, $K$ is a graded special $r$-submodule of $M$.
\end{proof}

\begin{lem}\label{Theorem 3.9} Suppose that $f:M\rightarrow M^{\prime}$ is a graded $R$-homomorphism.
\begin{enumerate}
\item If $f$ is a graded $R$-monomorphism and $K$ is a graded special $r$-submodule of $M^{\prime}$, then $f^{-1}(K)$ is a graded special $r$-submodule of $M$.
\item If $f$ is a graded $R$-epimorphism and $L$ is a graded special $r$-submodule of $M$ with $Ker(f)\subseteq L$, then $f(L)$ is a graded special $r$-submodule of $M^{\prime}$.
\end{enumerate}
\end{lem}

\begin{proof}
\begin{enumerate}
\item Let $a\in h(R)$ and $x\in h(M)$ such that $ax\in f^{-1}(K)$ with $Ann_{R}(x)=\{0\}$. Then $f(x)\in h(M^{\prime})$ such that $af(x)=f(ax)\in K$ with $Ann_{R}(f(x))=\{0\}$, and so $a\in (K:_{R}M^{\prime})\subseteq(f^{-1}(K):_{R}M)$ since $K$ is a graded special $r$-submodule of $M^{\prime}$. Hence, $f^{-1}(K)$ is a graded special $r$-submodule of $M$.
\item Let $a\in h(R)$ and $y\in h(M^{\prime})$ such that $ay\in f(L)$ with $Ann_{R}(y)=\{0\}$. Then there exists $x\in M$ such that $f(x)=y$, and since $Ann_{R}(y)=\{0\}$, $Ann_{R}(x)=\{0\}$. Now, $f(ax)=af(x)=ay\in f(L)$ and $Ker(f)\subseteq L$, hence $ax\in L$. Since $L$ is a graded special $r$-submodule of $M$, $a\in (L:_{R}M)\subseteq(f(L):_{R}M^{\prime})$. Thus, $f(L)$ is a graded special $r$-submodule of $M^{\prime}$.
\end{enumerate}
\end{proof}

\begin{thm}\label{Corollary 3.3} Let $L$ be a graded $R$-submodule of $M$. Then
\begin{enumerate}
\item for every graded special $r$-submodule $K$ of $M$ with $L\nsubseteq K$, $K\bigcap L$ is a graded special $r$-submodule of $L$.
\item for every graded special $r$-submodule $K$ of $M$ with $L\subseteq K$, $K/L$ is a graded special $r$-submodule of $M/L$.
\end{enumerate}
\end{thm}

\begin{proof}
\begin{enumerate}
\item Consider the injection $I:L\rightarrow M$, which is a graded $R$-monomorphism. Since $I^{-1}(K)=L\bigcap K$, $K\bigcap L$ is a graded special $r$-submodule of $L$ by Lemma \ref{Theorem 3.9} (1).
\item Consider the natural homomorphism $f:M\rightarrow M/L$, which is a graded $R$-epimorphism with $Ker(f)=L\subseteq K$. Since $f(K)=K/L$, $K/L$ is a graded special $r$-submodule of $M/L$ by Lemma \ref{Theorem 3.9} (2).
\end{enumerate}
\end{proof}

\begin{thm}\label{Theorem 3.10} Let $K$ be a proper graded $R$-submodule of a graded $R$-module $M$. Then $K$ is a graded special $r$-submodule of $M$ if and only if whenever $N$ is a graded $R$-submodule of $M$ such that $N\bigcap(h(M)-T(M))\neq\emptyset$ and $I$ is a graded ideal of $R$ such that $IN\subseteq K$, then $I\subseteq(K:_{R}M)$.
\end{thm}

\begin{proof} Suppose that $K$ is a graded special $r$-submodule of $M$. Let $N$ be a graded $R$-submodule of $M$ such that $N\bigcap(h(M)-T(M))\neq\emptyset$ and $I$ be a graded ideal of $R$ such that $IN\subseteq K$. Assume that $r\in I$. Then $r_{g}\in I$ for all $g\in G$ ,$I$ is graded. Since $N\bigcap(h(M)-T(M))\neq\emptyset$, there exists $x\in h(M)$ such that $Ann_{R}(x)=\{0\}$ and $x\in N$, and then $r_{g}x\in K$ for all $g\in G$. Using the fact that $K$ is a graded special $r$-submodule of $M$, we get $r_{g}\in (K:_{R}M)$ for all $g\in G$, and so $r=\displaystyle\sum_{g\in G}r_{g}\in (K:_{R}M)$. Hence, $I\subseteq(K:_{R}M)$. Conversely, let $r\in h(R)$ and $x\in h(M)$ such that $rx\in K$ with $Ann_{R}(x)=\{0\}$. Then $N=Rx$ is a graded $R$-submodule of $M$ such that $N\bigcap(h(M)-T(M))\neq\emptyset$ and $I=\langle r\rangle$ is a graded ideal of $R$ such that $IN\subseteq K$, and so $I\subseteq(K:_{R}M)$, which implies that $r\in (K:_{R}M)$. Therefore, $K$ is a graded special $r$-submodule of $M$.
\end{proof}

\begin{cor}\label{Theorem 3.11} Let $N$ be a graded $R$-submodule of $M$ such that $N\bigcap(h(M)-T(M))\neq\emptyset$.
\begin{enumerate}
\item If $K_{1}$ and $K_{2}$ are graded special $r$-submodules of $M$ such that $(K_{1}:_{R}M)N=(K_{2}:_{R}M)N$, then $(K_{1}:_{R}M)=(K_{2}:_{R}M)$.
\item If $IN$ is a graded special $r$-submodule of $M$ for some graded ideal $I$ of $R$, then $IN=IM$. In particular, $IM$ is a graded special $r$-submodule of $M$.
\end{enumerate}
\end{cor}

\begin{thm}\label{Theorem 3.10 (2)} Let $K$ be a proper graded $R$-submodule of a graded $R$-module $M$ such that $(h(M)\cap K) \subseteq T(M)$. If $K$ is not a graded special $r$-submodule of $M$ then there exist a graded ideal $I$ of $R$ and a graded $R$-submodule $N$ of $M$ such that $N\bigcap(h(M)-T(M))\neq\emptyset$, $K\subsetneqq N$, $(K:_{R}M)\subsetneqq I$ and $IN\subseteq K$.
\end{thm}

\begin{proof} Since $K$ is not a graded special $r$-submodule of $M$, there exist $r\in h(R)$ and $x\in h(M)$ such that $rx\in K$ with $Ann_{R}(x)=\{0\}$ but $r\notin (K:_{R}M)$. So, $N=(K:_{M}r)=\left\{m\in M:rm\in K\right\}$ is a graded $R$-submodule of $M$ such that $K\subsetneqq N$ since $x\in N-K$. Also, $I=(K:_{R}N)$ is a graded ideal of $R$ such that $(K:_{R}M)\subsetneqq I$ since $r\in I-(K:_{R}M)$, and then $IN=(K:_{R}N)N\subseteq K$.
\end{proof}

The following always holds.

\begin{lem}\label{Remark 3.1} Let $M$ be an $R$-module and $K$ be an $R$-submodule of $M$. Then $((K:_{M}r):_{R}M)=((K:_{R}M):_{R}r)$ for all $r\in R$.
\end{lem}

\begin{thm}\label{Proposition 3.13} Let $K$ be a graded special $r$-submodule of a graded $R$-module $M$ and $r\in h(R)-(K:_{R}M)$. Then $(K:_{M}r)$ is a graded special $r$-submodule of $M$.
\end{thm}

\begin{proof} Let $s\in h(R)$ and $x\in h(M)$ such that $sx\in (K:_{M}r)$ with $Ann_{R}(x)=\{0\}$. Then $rsx\in K$, and so $rs\in (K:_{R}M)$ since $K$ is a graded special $r$-submodule of $M$. Hence $s\in ((K:_{R}M):_{R}r)=((K:_{M}r):_{R}M)$ by Lemma \ref{Remark 3.1}. Thus, $(K:_{M}r)$ is a graded special $r$-submodule of $M$.
\end{proof}

\begin{cor}\label{Proposition 3.13 (2)} Let $M$ be a graded $R$-module and $r\in h(R)-Ann_{R}(M)$. Then $Ann_{M}(r)$ is a graded special $r$-submodule of $M$.
\end{cor}

\begin{proof} Apply Theorem \ref{Proposition 3.13} and Example \ref{Proposition 3.11}.
\end{proof}

\begin{thm}\label{Theorem 3.13 (2)} Let $M$ be a graded $R$-module. If $\{0\}$ is the only graded special $r$-submodule of $M$ and $M$ is homogeneous faithful, then $M$ is homogeneous torsion free.
\end{thm}

\begin{proof} Let $r\in h(R)$ and $x\in h(M)$ such that $rx=0$. If $r=0$, then we are done. Suppose that $r\neq0$. Then $r\notin Ann_{R}(M)$ since $M$ is homogeneous faithful, and so $Ann_{M}(r)$ is a graded special $r$-submodule of $M$ by Corollary \ref{Proposition 3.13 (2)}. Since $Ann_{M}(r)=\{0\}$  we we have $x=0$. Hence, $M$ is homogeneous torsion free.
\end{proof}

\begin{thm}\label{Proposition 3.14} If $K$ is a graded maximal special $r$-submodule of a graded $R$-module $M$, then $K$ is a graded prime $R$-submodule of $M$.
\end{thm}

\begin{proof} Let $r\in h(R)$ and $x\in h(M)$ such that $rx\in K$. Suppose that $r\notin (K:_{R}M)$. Then $(K:_{M}r)$ is a graded special $r$-submodule of $M$ by Theorem \ref{Proposition 3.13}, and so $x\in (K:_{M}r)= K$ since $K$ is a graded maximal special $r$-submodule of $M$. Thus, $K$ is a graded prime $R$-submodule of $M$.
\end{proof}

\begin{thm}\label{Theorem 3.14} Let $M$ be a graded $R$-module. Then every proper graded $R$-submodule of $M$ is graded special $r$-submodule if and only if $h(M)\subseteq T(M)$ or $Rx=M$ for all $x\in h(M)-T(M)$.
\end{thm}

\begin{proof} Suppose that every proper graded $R$-submodule of $M$ is graded special $r$-submodule. Assume that $h(M)\nsubseteq M$. Let $x\in h(M)-T(M)$. If $Rx\neq M$, then $Rx$ is a graded special $r$-submodule of $M$ by assumption. Let $r\in R$. Then $r=\displaystyle\sum_{g\in G}r_{g}$ where $r_{g}\in R_{g}$ for all $g\in G$, and then $r_{g}x\in Rx$ for all $g\in G$, which implies that $r_{g}\in (Rx:_{R}M)$ for all $g\in G$, and hence $r\in (Rx:_{R}M)$. So, $(Rx:_{R}M)=R$. Thus, $Rx=RM=M$ which is a contradiction. Therefore $Rx=M$ for all $x\in h(M)-T(M)$. Conversely, if $h(M)\subseteq T(M)$, then every proper graded $R$-submodule of $M$ is graded special $r$-submodule. Suppose that $Rx=M$ for all $x\in h(M)-T(M)$. Let $K$ be a proper graded $R$-submodule of $M$. Assume that $r\in h(R)$ and $x\in h(M)$ such that $rx\in K$ with $Ann_{R}(x)=\{0\}$. Then $Rx=M$, and so $r\in (K:_{R}x)=(K:_{R}M)$. Hence, $K$ is a graded special $r$-submodule of $M$.
\end{proof}

\end{document}